\documentclass{amsart}
\usepackage{hyperref}
\usepackage{url}
\usepackage{amsthm}
\usepackage{amsmath}
\usepackage{amssymb}
\usepackage{xcolor}
\usepackage{multicol}
\usepackage{systeme}
\usepackage{enumerate}
\usepackage{blkarray}
\usepackage[ruled,vlined]{algorithm2e}
\usepackage{authblk}
\usepackage[letterpaper, margin=1in]{geometry}
\usepackage[utf8]{inputenc}
\usepackage[T1]{fontenc}
\usepackage[english]{babel}

\newtheorem{theorem}{Theorem}[section]
\newtheorem{lemma}{Lemma}[section]
\newtheorem{proposition}{Proposition}[section]
\theoremstyle{definition}
\newtheorem{definition}{Definition}[section]

\newtheorem{corollary}{Corollary}[section]
\theoremstyle{definition}
\newtheorem{example}{Example}[section]

\numberwithin{equation}{section}

\hypersetup{backref=true,       
    pagebackref=true,               
    hyperindex=true,                
    colorlinks=true,                
    breaklinks=true,                
    urlcolor= black,                
    linkcolor= red,                
    bookmarks=true,                 
    bookmarksopen=false,
    filecolor=black,
    citecolor=red,
    linkbordercolor=blue}

\AtBeginDocument{\hypersetup{pdfborder={0 0 1}}}

\usepackage{etoolbox}
\patchcmd{\section}{\scshape}{\bfseries}{}{}
\makeatletter
\renewcommand{\@secnumfont}{\bfseries}
\makeatother

\title{Quadratization of ODE\lowercase{s}: Monomial vs. Non-Monomial}

\author[Alauddin]{Foyez Alauddin}
\address{$^1$Trinity School NYC, 101 West 91st St, New York, NY, 10024, USA} 
\email{foyez.alauddin21@trinityschoolnyc.org}
\author[Pogudin]{Advisor: Gleb Pogudin}
\address{$^2$LIX, CNRS, \'Ecole  Polytechnique,  Institute  Polytechnique  de  Paris,  Palaiseau,  91120, France} 
\email{gleb.pogudin@polytechnique.edu}

\begin{document}
\maketitle
\begin{center}
Foyez Alauddin$^1$ \\
Advisor: Gleb Pogudin$^2$ 
\end{center}

\begin{abstract}
    Quadratization is a transform of a system of ODEs with polynomial right-hand side into a system of ODEs with at most quadratic right-hand side via the introduction of new variables. It has been recently used as a pre-processing step for new model order reduction methods, so it is important to keep the number of new variables small. Several algorithms have been designed to search for a quadratization with the new variables being monomials in the original variables. To understand the limitations and potential ways of improving such algorithms, we study the following question: can quadratizations with not necessarily monomial new variables produce a model of substantially smaller dimension than quadratization with only monomial new variables?
    
    To do this, we restrict our attention to scalar polynomial ODEs. Our first result is that a scalar polynomial ODE $\dot{x}=p(x)=a_nx^n+a_{n-1}x^{n-1}+\ldots + a_0$ with $n\geqslant 5$ and $a_n\neq0$ can be quadratized using exactly one new variable if and only if $p(x-\frac{a_{n-1}}{n\cdot a_n})=a_nx^n+ax^2+bx$ for some $a,b \in \mathbb{C}$. In fact, the new variable can be taken $z:=(x-\frac{a_{n-1}}{n\cdot a_n})^{n-1}$. Our second result is that two non-monomial new variables are enough to quadratize all degree $6$ scalar polynomial ODEs. Based on these results, we observe that a quadratization with not necessarily monomial new variables can be much smaller than a monomial quadratization even for scalar ODEs.
    
    The main results of the paper have been discovered using computational methods of applied nonlinear algebra (Gr\"obner bases), and we describe these computations.
\end{abstract}




\section{Introduction}
    Model order reduction of non-linear dynamical systems is an important tool in applied mathematics. 
    The goal of reducing dynamical systems is to make them easier to analyze. 
    In this paper, we investigate quadratization, a technique used as a preprocessing step to some recent model order reduction methods. 

    Quadratization is a transformation of a system of ordinary differential equations with polynomial right-hand side into a system with at most quadratic right-hand side via the introduction of new  variables. 
    While quadratization does indeed lift up the dimension of a system, there are more powerful model order reduction methods for ODEs with at most quadratic right-hand side that produce better reductions than general methods~\cite{QLMOR, POD}. 
    In essence, quadratization lifts up the dimension of a system for it be pulled down by dedicated model order reduction methods. 
    
    We illustrate quadratization using the following simple scalar polynomial ODE: 
    \begin{equation}\label{initial}
        \dot{x}=x^{10}
    \end{equation}
    The right-hand side has degree greater than two, but by introducing $z:=x^9$, we can write:
    \begin{equation}\label{system}
    \begin{cases}
    \dot{x}=zx\\
    \dot{z}=9x^8\dot{x}=9x^{18}=9z^2
    \end{cases}
    \end{equation}
    In \eqref{system}, the right-hand side of $\dot{x}$ and $\dot{z}$ are at most quadratic.
    Furthermore, every solution of~\eqref{initial} yields a solution of~\eqref{system}. 
    We say that the order of a quadratization is the number of new variables introduced. 
    Thus, this quadratization has order $1$.
    
    In~\eqref{system}, our new variable $z$ is a monomial in $x$.
    Thus, we call this a monomial quadratization. 
    On the other hand, if one of our new variables has non-monomial right-hand side, we will refer to this as non-monomial quadratization. 
    Several algorithms have been developed to find monomial quadratizations of ODEs and ODE systems~\cite{Pogudin, Carothers, CRN}. 
    An additional motivation to study monomial quadratization comes from the desire to create realistic chemical reaction networks (CRNs), which can be interpreted as polynomial ODEs~\cite{CRN}.
    
    However, we are not aware of any algorithms for finding an optimal quadratization with not necessarily monomial new variables.
    In order to understand the potential benefits of such algorithms, we ask the following question: can non-monomial quadratizations produce a system of substantially smaller dimension than monomial quadratizations? 
    
    In order to approach this question, we consider scalar polynomial ODEs as it is ``the simplest nontrivial case''. 
    While model order reduction techniques are not useful for scalar polynomial ODEs as their dimension is already minimized at $1$, they allow us a point of entry to understand the question posed above.
    
    In our research, we completely characterized the case when one new variable is enough to quadratize a scalar polynomial ODE. 
    We also found that any degree $6$ scalar polynomial ODE can be quadratized with two non-monomial new variables. 
    Our main results are formally stated as the following:
    \begin{itemize}
        \item Theorem \textcolor{red}{\ref{theorem}}: Suppose
        \begin{equation*}
            p(x)=a_nx^n+a_{n-1}x^{n-1}+a_{n-2}x^{n-2}+\dots a_2x^2+a_1x+a_0
        \end{equation*}
        where $n\geqslant 5$, $a_i\in \mathbb{C}$ for all $i\in\{0,1,2\dots, n-1, n\}$, and $a_n\neq 0$. 
        A scalar polynomial ODE $\dot{x} = p(x)$ can be quadratized using exactly one new variable if and only if $p(x-\frac{a_{n-1}}{n\cdot a_n}) = a_nx^n+ax^2+bx$ for some $a,b\in\mathbb{C}$. 
        Moreover, this new variable can be taken to be $z:=(x-\frac{a_{n-1}}{n\cdot a_n})^{n-1}$.
        \item Theorem \textcolor{red}{\ref{degree 6}}: Suppose 
        \begin{equation*}
            \dot{x}=p_6x^6+p_5x^5+p_4x^4+p_3x^3+p_2x^2+p_1x+p_0
        \end{equation*}
        for $p_i\in\mathbb{C}$ for $i\in\{0,1,2,3,4, 5, 6\}$ and $p_6\neq 0$.  
        This scalar polynomial ODE can always be quadratized with two new variables.
        Moreover, they can be taken to be:
        \begin{equation*}
         z_1:=(\tfrac{x}{\sqrt[6]{p_6}}-\tfrac{p_5}{6\cdot p_6})^5+(\tfrac{25p_5^3}{216\sqrt{p_6^5}}-\tfrac{5p_5p_4}{12\sqrt{p_6^3}}+\tfrac{5p_3}{8\sqrt{p_6}})(\tfrac{x}{\sqrt[6]{p_6}}-\tfrac{p_5}{6\cdot p_6})^3,\qquad z_2:=(\tfrac{x}{\sqrt[6]{p_6}}-\tfrac{p_5}{6\cdot p_6})^2.
        \end{equation*}
    \end{itemize}
    We also show that all degree $3$ and $4$ scalar polynomial ODEs can be quadratized with one monomial new variable and all degree $5$ scalar polynomial ODEs can be quadratized with two monomial new variables (see Proposition \ref{3,4,5}). Each of the quadratizations presented in our main results are optimal. 
    
    In order to acheive these results, we employed computational techniques that made use of Gröbner basis (see section \ref{section: Computational Techniques}). 
    In Section \ref{section: Proofs and Other Results}, we provide mathematical proofs of Theorem~\ref{theorem} and Theorem \ref{degree 6}.



\section{Preliminaries}
    \begin{definition}
    Consider the following scalar polynomial ODE:
    \begin{equation}
        \dot{x}=p(x)
    \end{equation} 
    where $p(x) \in \mathbb{C}[x]$. Then, a list of $m$ new variables:
    \begin{equation}
        z_1:=z_1(x), z_2:=z_2(x),\dots, z_m:=z_m(x)
    \end{equation}
    is said to \emph{quadratize} $\dot{x}$ if there exists polynomials $h_1, h_2, \dots, h_{m+1} \in \mathbb{C}[x, z_1, z_2, \dots, z_m]$ of degree at most two such that:
    \begin{itemize}
        \item $\dot{x}=h_1(x, z_1, z_2, \dots, z_m)$;
        \item $\dot{z_i}=z_i'(x)\dot{x}=h_{i+1}(x, z_1, z_2, \dots, z_m)$ for $1\leqslant i \leqslant m$
    \end{itemize}
    
    The number $m$ is said to be the \emph{order} of the quadratization. A quadratization of the smallest possible order is called \emph{optimal}. We refer to new variables whose right-hand sides are monomials as \emph{monomial new variables} and those whose right-hand sides are not monomials as \emph{non-monomial new variables}.
    \end{definition}

    \begin{example}
        Consider the scalar polynomial ODE $\dot{x}=x^n$ with $n > 2$. 
        Let $z:= x^{n-1}$. Note that $z$ is a monomial new variable. We will use $z$ to quadratize $\dot{x}$. We can write:
        \begin{equation}
            \begin{cases}
            \dot{x}=zx\\
            \dot{z}=z'(x)\dot{x} =  (n-1)x^{n-2} \cdot x^n=(n-1)x^{2n-2}=(n-1)z^2
            \end{cases}
        \end{equation}
        Thus, we have quadratized $\dot{x}$ with $z:=x^{n-1}$ as both $\dot{x}$ and $\dot{z}$ are can be written as quadratic polynomials in $z$ and $x$. In particular, this quadratization has order $1$ and is optimal. 
    \end{example}
    
    \begin{example}
        Consider the scalar polynomial ODE $\dot{x}=x^5+x^4+x^3+x^2+x+1$. We let $z_1(x):=x^4$ and $z_2(x):=x^3$. Note that $z_1$ and $z_2$ are monomial new variables. It follows that:
        \begin{equation}
        \begin{cases}
            \dot{x}=z_1x+z_1+z_2+x^2+x+1\\
            \dot{z}_1=z_1'(x)\dot{x}=4x^3\dot{x}=4(z_1^2+z_1z_2+z_2^2+z_1x+z_1+z_2)\\
            \dot{z}_2=z_2'(x)\dot{x}=3x^2\dot{x}=3(z_1z_2+z_2^2+z_1x+z_1+z_2+x^2)
        \end{cases}
        \end{equation}
        Thus, we have quadratized $\dot{x}$ with $z_1$ and $z_2$ as $\dot{x}$, $\dot{z}_1$, and $\dot{z}_2$ are written as quadratic polynomials in $z_1, z_2$, and $x$. In particular, this quadratization has order $2$. It can be shown using Theorem \ref{theorem} and Proposition \ref{3,4,5} that this quadratization is optimal. 
    \end{example}
    


\section{Main Results}

    Our main results are Theorem \textcolor{red}{\ref{theorem}} and Theorem \textcolor{red}{\ref{degree 6}}. 
    \begin{theorem}\label{theorem}
        Suppose
        \begin{equation*}
            p(x)=a_nx^n+a_{n-1}x^{n-1}+a_{n-2}x^{n-2}+\dots a_2x^2+a_1x+a_0
        \end{equation*}
        where $n\geqslant 5$, $a_i\in \mathbb{C}$ for all $i\in\{0,1,2\dots, n-1, n\}$, and $a_n\neq 0$. 
        A scalar polynomial ODE $\dot{x} = p(x)$ can be quadratized using exactly one new variable if and only if $p(x-\frac{a_{n-1}}{n \cdot a_n}) = a_nx^n+ax^2+bx$ for some $a,b\in\mathbb{C}$. 
        Moreover, this new variable can be taken to be $z:=(x-\frac{a_{n-1}}{n \cdot a_n})^{n-1}$.
    \end{theorem}
    
    In Theorem \textcolor{red}{\ref{theorem}}, the quadratization is optimal as it has order one, and the original ODE has not been quadratic already.
    
    \begin{theorem}\label{degree 6}
        Suppose 
        \begin{equation*}
            \dot{x}=p_6x^6+p_5x^5+p_4x^4+p_3x^3+p_2x^2+p_1x+p_0
        \end{equation*}
        for $p_i\in\mathbb{C}$ for $i\in\{0,1,2,3,4, 5, 6\}$ and $p_6\neq 0$.  Then, $\dot{x}$ can be quadratized with two new variables of the form:
        \begin{equation*}
         z_1:=(\tfrac{x}{\sqrt[6]{p_6}}-\tfrac{p_5}{6\cdot p_6})^5+(\tfrac{25p_5^3}{216\sqrt{p_6^5}}-\tfrac{5p_5p_4}{12\sqrt{p_6^3}}+\tfrac{5p_3}{8\sqrt{p_6}})(\tfrac{x}{\sqrt[6]{p_6}}-\tfrac{p_5}{6\cdot p_6})^2, \qquad z_2:= (\tfrac{x}{\sqrt[6]{p_6}}-\tfrac{p_5}{6\cdot p_6})^3.
        \end{equation*}
    \end{theorem}
    
    Notice that $\dot{x}$ represents any degree $6$ scalar polynomial ODE. 
    For equations not satifying the requirements of Theorem \ref{theorem}, this quadratization is optimal. 
    
    Additionally, we've shown that the general form degree $6$ scalar polynomial ODE cannot be quadratized with two monomial new variables, but can be quadratized with three monomial new variables (see Lemma \ref{cantquad} and Lemma \ref{canquad}).
    
    \begin{proposition}\label{3,4,5}
    \hfill
        \begin{enumerate}[(i)]
            \item All degree $3$ scalar polynomial ODEs can be quadratized by exactly one new variable, $z:=x^2$.
            \item All degree $4$ scalar polynomial ODEs can be quadratized by exactly one new variable, $z:=x^3$. 
            \item All degree $5$ scalar polynomial ODEs can be quadratized by exactly two new variables, $z_1:=x^4$ and $z_2:=x^3$.
        \end{enumerate}
    \end{proposition}

    Note that in parts $(i)$ and $(ii)$ of Proposition \ref{3,4,5}, the quadratizations are optimal for precisely the same reason the quadratization in Theorem \ref{theorem} is optimal. 



\section{Discussion}
    In our main results, we have given conditions for when one new variable is enough to quadratize a scalar polynomial ODE. In fact, this new variable has non-monomial right-hand side. We have also shown that two non-monomial new variables is enough to quadratize any degree $6$ scalar polynomial ODE. 
    
    Theorem \textcolor{red}{\ref{theorem}} is interesting because it demonstrates that even if we deal with high degree scalar polynomial ODEs, there is a certain form of these ODEs that can be quadratized with only one non-monomial new variable. In particular, the most interesting part of Theorem \textcolor{red}{\ref{theorem}} is its use of linear shift. We consider the following scalar polynomial ODE to illustrate this important feature of Theorem \textcolor{red}{\ref{theorem}}:
    \begin{equation}\label{example}
        \dot{x}=(x+1)^n=x^n+{n\choose 1}x^{n-1}+{n\choose 2}x^{n-2}+\dots+{n\choose{n-2}}x^2+{n\choose{n-1}}x+1
    \end{equation}
    In \eqref{example}, every coefficient behind $x^k$ for $k\in\{0, 1, 2, \dots, n\}$ is non-zero. While it doesn't appear to be quadratizable by just one new variable, Theorem \textcolor{red}{\ref{theorem}} tells us it is. We can say that $p(x)=(x+1)^n$ and $a_{n-1}=n$. Thus, it follows that: 
    \begin{equation}
        p(x-\frac{a_{n-1}}{n})=p(x-1)=(x-1+1)^n=x^n
    \end{equation}
    By Theorem \ref{theorem}, since $x^n$ is of the form $x^n+ax^2+bx$, it can be quadratized by one new variable. 
    
    Most algorithms do not consider this shift in determining the order of quadratization and thus, for arbitrarily large $n$, these algorithms would introduce more variables than necessary. 
    
    More precisely, we can provide a lower bound on the number of monomial new variables necessary to quadratize the scalar ODE presented in \eqref{example}. 
    Suppose set $S=\{z_i\ |\ 1\leqslant i \leqslant k\}$ where $2\leqslant \deg z_1 < \deg z_2 < \dots < \deg z_k$ denotes the set of monomial new variables used to quadratize the ODE in \eqref{example} (we do not consider degree less than two due to Lemma~\ref{claim}).
    Let us also append $x$ and $1$ to set $S$. 
    Any quadratic term in our quadratization can be formed by choosing any two, not necessarily distinct, elements of set $S$ and multiplying them together. 
    It follows that we can form at most ${{k+3}\choose{2}}$ quadratic monomial terms. 
    Since the right-hand side of~\eqref{example} must be quadratized and contains $n + 1$ monomials, we have:
    \begin{equation}
            {{k+3}\choose{2}}\geqslant n+1  \quad\implies\quad k\geqslant \frac{-5 + \sqrt{8n+9}}{2}
    \end{equation}
    Here, we have provided a lower bound for $k$ or the number of new variables introduced. For larger $n$, we get larger $k$. However, we show that simply one new variable is enough for any value of $n$ if we consider non-monomial new variables, which is a significant improvement on monomial quadratization. 
    
    Similarly, Theorem \textcolor{red}{\ref{degree 6}} shows that degree $6$ scalar polynomial ODE can be quadratized by two non-monomial new variables. 
    On the other hand, we also showed in Lemma~\ref{cantquad} and Lemma~\ref{canquad} that a general degree $6$ scalar polynomial ODE cannot be quadratized using two monomial new variables, but all degree $6$ scalar polynomial ODEs can be quadratized using three monomial new variables. By allowing our new variables to have non-monomial right-hand side, we improve the order of the quadratization by $~30\%$. 
    
    Altogether, our results suggest that considering quadratizations with non-monomial new variables can give us more optimal quadratizations than monomial quadratizations as we increase the degree  and dimension of our ODEs. 
    Therefore, considering non-monomial quadratization of multivariable ODEs and ODE systems may be a worthwhile pursuit. 
    Moreover, since our results provide explicit formulas for new variables, they can be used to improve current algorithms for monomial quadratization, for example, by applying some variable shifts to the input system as suggested by Theorem~\ref{theorem}.



\section{Proofs and Other Results}\label{section: Proofs and Other Results}
    The following three lemmas are used to prove Theorem \textcolor{red}{\ref{theorem}}. 
    
    \begin{lemma}\label{lemma:shift}
        For every scalar polynomial ODE's $\dot{x}=a_nx^n+a_{n-1}x^{n-1}+...+a_2x^2+a_1x+a_0$, there exists a unique change of variables $x \to x + \lambda$ such that $a_{n-1}$ becomes zero.
    \end{lemma}
    
    \begin{proof}
        Let $x=y+\lambda$. Substituting for $x$ in $\dot{x}$, we get:
        \begin{equation*}
            \dot{y}=a_n(y+\lambda)^n+a_{n-1}(y+\lambda)^{n-1}+\dots +a_2(y+\lambda)^2+a_1(y+\lambda)+a_0
        \end{equation*}
        Binomially expanding, we have:
        \begin{equation*}
            \dot{y} = a_ny^n + (a_nn\lambda+a_{n-1}) y^{n-1} + o(y^{n-1})
        \end{equation*}
        Since the coefficient behind the $y^{n-1}$ is $a_nn\lambda+a_{n-1}$, it follows that the $y^{n-1}$ term vanishes if and only if  $\lambda=\frac{-a_{n-1}}{n\cdot a_n}$. 
        Since $n>0$ and $a_n\neq 0$, this shift exists and is unique.
    \end{proof}
    
    \begin{lemma}\label{variable}
        Assume that a scalar polynomial ODE $\dot{x}=a_nx^n + q(x)$ with $n\geqslant 5$ and $\deg q(x) \leqslant n - 1$ can be quadratized by a single new variable $z:=z(x)$.
        Then, $\deg z = n - 1$.
    \end{lemma}

    \begin{proof}
        The first term of $\dot{x}$, $x^n$, must be quadratized by a quadratic term in $x$ and $z$.
        For $n \geqslant 5$, only terms $z$, $xz$, and $z^2$ may involve $x^n$, and this may happen only if $\deg z \geqslant 3$.
        Thus, $\deg z < \deg xz < \deg z^2$. 
        Hence $x^n$ must be the leading monomial of one of them.
        Thus, $\deg z \in \{n, n - 1, n / 2\}$.
          
        If $\deg z = n$, then $\deg \dot{z} = 2n - 1$.
        Since $\deg z^2 > 2n - 1$ and the degree of any other quadratic monomial in $x$ and $z$ is less than $2n - 1$, $\dot{z}$ cannot be quadratized. Thus, $\deg z \neq n$. 
          
        If $n$ is odd, the only remaining option is $\deg z = n - 1$, so we are done.
        Consider the case of even $n$ and $\deg z = \frac{n}{2}$.
        Let $z=z(x):= \alpha x^{\frac{n}{2}} + r(x)$ and $\dot{x} = x^n + q(x)$ where $\deg r(x) < n / 2$, $\deg q(x) <  n$, and $\alpha \neq 0$. 
        Since $\dot{z}$ must be quadratized, we have that:
        \begin{equation*}\label{eq: dotz 1}
            \dot{z} = z'(x) \dot{x} = \left(\alpha\frac{n}{2}x^{\frac{n-2}{2}} + \mathrm{o}(x^{\frac{n-2}{2}}) \right) (x^n + \mathrm{o}(x^n)) = \alpha\frac{n}{2}x^{\frac{3n - 2}{2}} + \mathrm{o}(x^{\frac{3n - 2}{2}}).
        \end{equation*}
        Since $\deg z \geqslant 3$, the degree of any quadratic polynomial in $x$ and $z$ is at most $2\deg z \leqslant n$, which is less than $\frac{3n - 2}{2}$ for $n \geqslant 6$. So, $\deg z \neq \frac{n}{2}$. Thus, $\deg z = n-1$. 
    \end{proof}
    
    \begin{lemma}\label{claim}
        Suppose $z_1, z_2, \dots, z_m$ quadratize some $\dot{x} = p(x)$. 
        Then, the same new variables with omitted constant and linear (w.r.t. $x$) terms also quadratize this ODE. 
    \end{lemma}
    \begin{proof}
        For each $i$ in $\{1,2,\dots, m\}$, let:
        \begin{equation*}
            z_i:=a_ix^{k_i} + \dots 
            + b_ix + c_i
        \end{equation*}
        Since $z_1, z_2, \dots, z_m$ quadratize $\dot{x}$, it follows that $\dot{x}, \dot{z}_1, \dot{z}_2, \dots, \dot{z}_m$ are written with at most quadratic right-hand side in $x, z_1, z_2, \dots, z_m$. For $i\in \{1, 2, \dots, m\}$, any quadratic terms in $z_i$'s can be written as quadratic in $z_i-b_ix-c_i$'s and $x$. 
    \end{proof}
    
    \begin{proof}[Proof of Theorem \ref{theorem}]\label{prooftheorem}
    We will first prove the backward direction: if $p(x-\frac{a_{n-1}}{n \cdot a_n}) = a_nx^n+ax^2+bx$ for some $a,b\in\mathbb{C}$, then $\dot{x}=p(x)$ can be quadratized using exactly one new variable. 
    
    Suppose $\dot{x}=p(x)=a_nx^n+a_{n-1}x^{n-1}+o(x^{n-1})$ and $p(x-\frac{a_{n-1}}{n\cdot a_n})=a_nx^n+ax^2+bx$. 
    So, we will shift $\dot{x}$ with the change of variables $x = y - \frac{a_{n-1}}{n\cdot a_n}$. Substituting for $x$ in $\dot{x}$, we have that $\dot{y} = a_ny^n+ay^2+by$ for some $a,b \in \mathbb{C}$.
    Let $z:=y^{n-1}$.
    It follows that:
    \begin{equation*} \label{eq: dotx 2}
        \begin{cases}
            \dot{y}=a_nzy+ay^2+by\\
            \dot{z}=(n-1)y^{n-2}(\dot{x})=(n-1)(a_ny^{2n-2}+ay^{n}+by^{n-1})=(n-1)(a_nz^2+azy+bz)
        \end{cases}
    \end{equation*}
    
    Now, we will prove the forward direction: if a scalar polynomial ODE $\dot{x} = p(x)$ can be quadratized using exactly one new variable, then $p(x-\frac{a_{n-1}}{n \cdot a_n}) = a_nx^n+ax^2+bx$ for some $a,b\in\mathbb{C}$.  
    
    Shifting $x$ as described in Lemma \ref{lemma:shift}, we will assume in what follows that our ODE is of the form
    \begin{equation*} \label{eq: dotx 3}
        \dot{x}=a_nx^n+q(x) 
    \end{equation*}
    where $\deg q(x)\leqslant n-2$ and $a_n \neq 0$. 
    By Lemma \ref{variable}, our new variable $z$ must be of degree $n-1$. Thus, let:
    \begin{equation*} \label{eq: dotz 3}
        z:=x^{n-1}+r(x)
    \end{equation*}
    By Lemma \ref{claim}, we can take $r(x)$ with no linear or constant term. Since $\dot{x}$ must be quadratized by $z$, we can write:
    \begin{equation} \label{eq: dotx 4}
        \dot{x}=a_nx^n+q(x)=a_nxz+ez+ax^2+bx+c=a_nx^n+(a_nx + e)r(x)+ax^2+bx+c
    \end{equation}
    Notice that $z^2$ is not involved in \eqref{eq: dotx 4} because for $n\geqslant 5, \deg z^2>\deg \dot{x}$. From \eqref{eq: dotx 4}, it follows that:
    \begin{equation} \label{eq: q(x) 2}
        q(x)=(a_nx + e)r(x)+ax^2+bx+c
    \end{equation}
    From \eqref{eq: q(x) 2}, since $\deg q(x)\leqslant n-2$, we observe that $2\leqslant d := \deg r(x) \leqslant n - 3$. This implies that $e=0$ because the $ez$ term in \eqref{eq: dotx 4} is the only term that involves $x^{n-1}$. However, we know that $\dot{x}$ has no $x^{n-1}$ term. 
    
    Since $\deg r(x) \leqslant n-3$, we can write:
    \begin{equation*}
        \begin{cases}
        r(x)=c_dx^d+c_{d-1}x^{d-1}+...+c_2x^2\\
        r'(x)=(d)c_dx^{d-1}+(d-1)c_{d-1}x^{d-2}+...+2c_2x
        \end{cases}
    \end{equation*}
    We assume that $r(x)$ is nonzero and in fact, $2\leqslant d \leqslant n-3$. Using the next two equations below, we will use proof by contradiction show that $r(x)=r'(x)=0$. 
    
    We look to write $\dot{z}$ in two different ways. The first way is by direct calculation:
    \begin{equation}\label{eq: dotz 4}
    \begin{aligned}
       \dot{z}=z'(x)\dot{x}&=((n-1)x^{n-2}+r'(x))(a_nx^n+a_nxr(x)+ax^2+bx+c)\\&=(n-1)x^{n-2}\cdot a_nx^n+(n-1)x^{n-2} \cdot a_nxr(x)+r'(x) \cdot a_nx^n+\mathrm{o}(x^{n-1+d})\\&=(n-1)a_nx^{2n-2}+(n-1)c_da_nx^{n-1+d}+(d)c_da_nx^{n-1+d} +\mathrm{o}(x^{n-1+d})
    \end{aligned}
    \end{equation}
    On the other hand, by definition of quadratization, the right-hand side of $\dot{z}$ must be written as at most quadratic in $x$ and $z$. Thus, $\dot{z}$ must have the form:
    \begin{equation} \label{eq: dotz 5}
        \begin{aligned}
        \dot{z}&=(n-1)a_nz^2+b_1zx+b_2z+b_3x^2+b_4x+b_5\\&=(n-1)a_nx^{2n-2}+2(n-1)a_nx^{n-1}r(x)+\mathrm{o}(x^{n-1+d})\\&=(n-1)a_nx^{2n-2}+2(n-1)c_da_nx^{n-1+d}+\mathrm{o}(x^{n-1+d})
        \end{aligned}
    \end{equation}
    for constants $b_1, b_2, b_3, b_4, b_5$. 
    
    Setting \eqref{eq: dotz 4} and \eqref{eq: dotz 5} equal to each other and simplifying, we obtain the following: 
    \begin{equation}\label{eq: r'(x) differential equation}
   (d)c_da_nx^{n-1+d} +\mathrm{o}(x^{n-1+d})=(n-1)c_da_nx^{n-1+d}+\mathrm{o}(x^{n-1+d})
    \end{equation}
    Analyzing the coefficients of the highest degree terms of each side of \eqref{eq: r'(x) differential equation}, we have: 
    \begin{equation}\label{eq: highest degree 2}
        (n-1)c_da_n=(d)c_da_n
    \end{equation}
    Since $c_d \neq 0$ and $a_n \neq 0$, we can divide both sides of \eqref{eq: highest degree 2} by them:
    \begin{equation}
        n-1=d
    \end{equation}
   However, since $d\leqslant n-3$,  we have reached a contradiction. So, $r(x)=r'(x)=0$.
    
    To complete the proof, we will show that the linear term $c$ in $\dot{x}$ must be zero. Using $r(x)=0$, we have:
    \[
     \dot{x} = a_nx^n + q(x), \quad q(x) = ax^2 + bx + c
    , \quad z:=x^{n-1}
    \]
    \begin{equation*} \label{dotx final}
    \dot{x} = a_nzx + ax^2 + bx + c
    \end{equation*}
    \begin{equation} \label{eq: dotz quadratization 1}
    \dot{z} = z'(x)\dot{x} = (n-1)a_nx^{2n-2}+(na-a)x^n+(nb-b)x^{n-1}+(nc-c)x^{n-2}
    \end{equation}
    Since $\dot{z}$ must be written as at most quadratic in $x$ and $z$, it must be of the form:
    \begin{equation}\label{z_prime_template}
      \dot{z} = (n-1)a_nz^2 + b_1zx + b_2z + b_3x^2 + b_4x + b_5
    \end{equation}
    for some $b_1, b_2, b_3, b_4, b_5 \in \mathbb{C}$. In \eqref{eq: dotz quadratization 1}, we have the term $(nc-c)x^{n-2}$. However, this term cannot be quadratized for $n\geqslant 5$ because there is no term with degree $n-2$ in \eqref{z_prime_template}. Thus, we find that $nc-c = 0$. Since $n \geq 5$, $c=0$. All other terms in \eqref{eq: dotz quadratization 1} can be written using some quadratic combination of $z$ and $x$.
    \end{proof}
    
    The following two lemmas and corollary are used to prove Theorem \ref{degree 6}.
    
    \begin{lemma}\label{simplify}
        Suppose $z_1, z_2, \dots, z_k$ quadratize $\dot{x}$. Consider $z_1$ and $z_2$. Let $a, b \in \mathbb{C}$. If $z_1, z_2, \dots, z_k$ quadratize $\dot{x}$, then $az_1+bz_2, z_2, \dots, z_k$ also quadratizes $\dot{x}$.
    \end{lemma}
    \begin{proof}
         Since $z_1, z_2, \dots, z_k$ quadratize $\dot{x}$, it follows that $\dot{x}, \dot{z}_1, \dot{z}_2, \dots, \dot{z}_k$ are written with at most quadratic right-hand side in $x, z_1, z_2, \dots, z_k$. Any quadratic term in $z_1$, $z_2$, \dots, $z_k$, and $x$ can be written as quadratic in $az_1+bz_2$, $z_2$, $z_3$, \dots, $z_k$, and $x$. Thus, it holds that if $z_1, z_2, \dots, z_k$ quadratize $\dot{x}$, then $az_1+bz_2, z_2, \dots, z_k$ also quadratizes $\dot{x}$.
    \end{proof}
    
    \begin{corollary}\label{cor: 2 var shift}
        Suppose $z_1, z_2, \dots z_k$ quadratize $\dot{x}$ where each term has leading coefficient $c_i$. Then, each of the new variables can have a distinct degree.
    \end{corollary}
    \begin{proof}
        The proof follows directly from Lemma \ref{simplify}. 
    \end{proof}
    
    \begin{lemma}\label{lemma: n-1}
        Let $\dot{x} = p(x)$ be a scalar polynomial ODE with $\deg p := n\geqslant 5$ that can be quadratized with $k\geqslant 2$ new variables. Let $z_1, z_2, z_3, \dots, z_k$ denote the $k$ new variables used to quadratize $\dot{x}$. Suppose that for any $i\in \{1,2,3,\dots, k\}$, $\deg z_i \leqslant n-1$. For $S=\{\deg z_1, \deg z_2, \deg z_3, \dots, \deg z_k\}$, it holds that $n-1\in S$. 
    \end{lemma}
    \begin{proof} 
        Assume for contradiction that $n-1 \notin S$. Thus, it follows that $\max(S)\leqslant n-2$. Assume that $\max(S)$ is $r$ for some $2 \leqslant r \leqslant n-2$. Let $z_r$ be the new variable with degree $r$. It follows that $\deg z_r = n+r-1$. Notice that the largest degree we can form with our new variables is $2r$. Thus, it follows that $2r\geqslant n+r-1$. However, this implies that $r\geqslant n-1$. Thus, $n-1\in S$. 
    \end{proof}
    
    \begin{lemma}\label{2 var limitation}
        Assume that $\dot{x}= p(x)$ be a scalar polynomial ODE with $\deg p := n\geqslant 5$ can be quadratized by two new variables. 
        Then, it can be quadratized using two new variables $y$ and $z$, one of which has degree $n-1$. 
    \end{lemma}
    \begin{proof}
        Let $y$ and $z$ be some quadratizing variables.
        By Corollary \ref{cor: 2 var shift}, they can be assumed to have distinct degrees. By Lemma \ref{claim},  $y$ and $z$ can be taken with no linear or constant term. So, we will assume $\deg y > \deg z \geqslant 2$. It holds that $\deg \dot{y} = n-1 + \deg y$. 
        
        We will first consider the case where $\deg y\leqslant n-1$. The proof of this follows directly from Lemma \ref{lemma: n-1}.
        
        Now, we will consider the case where $\deg y \geqslant n$. If $\deg y\geqslant n$, then $y^2$ cannot appear in the right-hand side of $\dot{y}$ because $\deg y^2 > \deg \dot{y}$. We will look to write the leading term of $\dot{y}$ as the highest degree term of some linear combination of the quadratic terms in $x,y,$ and $z$. 
        The leading term of $\dot{y}$ cannot be written as the highest degree term of any linear combination of $1, x, x^2, y, z, yx,$ or $zx$ because their degrees are too small. 
        
        Thus, we are left with $z^2$ and $yz$. Note that $\deg z^2 < \deg yz$. Thus, one of them must have degree $n-1+\deg y$. If $\deg yz = n-1+\deg y$, then it holds that $\deg z = n-1$.
        
        To finish the proof, we will show that $z^2$ cannot have degree $n-1 +\deg y$. 
        If $\deg z^2=n-1+\deg y$, then $\deg z \geqslant n$ and $\deg y \geqslant n+1$. Additionally, $\deg z^2=n-1+\deg y$ implies that the right-hand side must be even. So, $\deg y = n-1 + 2k$ for some $k\in\mathbb{N}$. It follows that $\deg z^2=2n-2+2k$, $\deg z = n-1+k$, and $\deg \dot{z} = 2n-2+k$. We must have that the leading term of $\dot{z}$ can be written as the highest degree term of some linear combination of the quadratic terms in $x, y,$ and $z$. No linear combination of $1, x, x^2, y, z, yx,$ and $zx$ have degree $2n-2+k$ because the degree of each of these terms is too small. Furthermore, no linear combination of $z^2, y^2,$ and $yz$ – each of which have distinct degrees – has degree $2n-2+k$ because the degree of each of these terms is too large. Thus, $z^2$ cannot have degree $n-1+\deg y$. 
    \end{proof}
 
    \begin{proof}[Proof of Theorem \ref{degree 6}]\label{proofdeg6}
         Suppose
         \begin{equation*}
             \dot{x} = p_6 x^6 + p_5 x^5 + p_4 x^4 + p_3x^3 + p_2 x^2 + p_1x +p_0
         \end{equation*}
         Let $x=\tfrac{y}{\sqrt[6]{p_6}}-\tfrac{p_5}{6 \cdot p_6}$. This change of variable aims to make $p_5$ equal to zero as demonstrated by Lemma \ref{lemma:shift} and make the leading coefficient $1$. Substituting for $x$, we have: 
         \begin{equation*}
         \begin{split}
             \dot{x} =\ & y^6 +  (\tfrac{p_4}{\sqrt[3]{p_6^2}}-\tfrac{5p_5^2}{12\sqrt[3]{p_6^5}})y^4+ (\tfrac{5p_5^3}{27\sqrt{p_6^5}}+\tfrac{p_3}{\sqrt{p_6}}-\tfrac{2p_5p_4}{3\sqrt{p_6^3}})y^3 + (\tfrac{p_5^2p_4}{6\sqrt[3]{p_6^7}}+\tfrac{p_2}{\sqrt[3]{p_6}}-\tfrac{p_5p_3}{2\sqrt[3]{p_6^4}}-\tfrac{5p_5^4}{144\sqrt[3]{p_6^{10}}})y^2 \\& + (\tfrac{p_5^5}{324\sqrt[6]{p_6^{25}}}+\tfrac{p_5^2p_3}{12\sqrt[6]{p_6^{13}}}-\tfrac{p_5^3p_4}{54\sqrt[6]{p_6^{19}}}-\tfrac{p_5p_2}{3\sqrt[6]{p_6^7}}+\tfrac{p_1}{\sqrt[6]{p_6}})x + (\tfrac{p_5^4p_4}{1296p_6^4} - \tfrac{5p_5^6}{46656p_6^5}-\tfrac{p_5^3p_3}{216p_6^3}-\tfrac{p_5p_1}{1296p_6^4}+p_0)
         \end{split}
         \end{equation*}
         For simplicity of notation, we will write:
         \begin{equation*}
         \dot{x}=q_0+q_1y+q_2y^2+q_3y^3+q_4y^4+y^6
         \end{equation*}
         Notice that $\tfrac{5q_3}{8} = \tfrac{25p_5^3}{216\sqrt{p_6^5}}-\tfrac{5p_5p_4}{12\sqrt{p_6^3}}+\tfrac{5p_3}{8\sqrt{p_6}}$. 
         By Lemma \ref{2 var limitation}, one of the new variables has degree $5$. Thus, let: 
        \begin{equation*}                        z_1:=y^5+\frac{5q_3}{8}y^2, \qquad     z_2: =y^3
        \end{equation*}
        It follows that for any constants $c_1, c_2, c_3\in\mathbb{C}$, we have that:
        \begin{equation}
        \begin{cases}
            \dot{x}=(1-c_1)z_1y+c_1z_2^2+(\tfrac{5c_1q_3+3q_3}{8}) z_2 +q_4z_2y+q_2y^2+q_1y+q_0\\
            \begin{split}
            \dot{z}_1=\  & 5z_1^2 + (5q_1-\tfrac{15q_3q_4}{8})z_1+(5q_2-c_2)z_1y+(\tfrac{5c_2q_3-15q_2q_3}{8})z_2+(5q_0-\tfrac{45q_3^2}{64})z_2y+c_2z_2^2\\&+5q_4z_1z_2+(\tfrac{75q_3^2q_4-120q_1q_3}{64})y^2+\tfrac{5q_0q_3}{4}y
            \end{split}\\
            \dot{z}_2 = 3z_1z_2+\frac{9q_3}{8}z_1+(3q_4-c_3)z_1y+(3q_1+\frac{5c_3q_3-15q_3q_4}{8})z_2+3q_2z_2y+c_3z_2^2+(3q_0-\frac{45q_3^2}{64})y^2
            \end{cases}
        \end{equation}
    \end{proof}
    Notice that the presence of constants $c_1, c_2, c_3$ in Theorem \textcolor{red}{\ref{degree 6}} suggests that there exists an infinite number of possible quadratizations of a degree $6$ scalar polynomial ODE with two new variables. 
    
    \begin{lemma}\label{cantquad}
        There exists scalar polynomial ODEs of degree $6$ which cannot be quadratized using two monomial new variables. 
    \end{lemma}
    \begin{proof}
        Consider
        \begin{equation*}\label{general}
            \dot{x}=p_6x^6 + p_4x^4 + p_3x^3 + p_2x^2 + p_1x + p_0
        \end{equation*}
        where $p_6, p_4, p_3, p_2, p_1, p_0 \in \mathbb{C}\backslash \{ 0 \}$. We use this form for $\dot{x}$ to reflect the possible use of shift on a general form degree $6$ scalar polynomial ODE as discussed in Lemma \ref{lemma:shift}.
        By Lemma \ref{2 var limitation}, one of the new variables must have degree $5$. Thus, for monomial quadratization, we have the following cases:
        \begin{equation*}
            \begin{cases}
                Case\ 1: z_1:= x^5, z_2:=x^2\\
                Case\ 2: z_1:=x^5, z_2:=x^3\\
                Case\ 3: z_1:=x^5, z_2:=x^4 \\
                Case\ 4: z_1:=x^5, z_2:=x^{5+k}\ \text{for}\ k \in \mathbb{N}
            \end{cases}
        \end{equation*}
        Notice that in Cases $1, 2$, and $3$, $\max \{\deg \dot{x}, \deg \dot{z}_1, \deg \dot{z}_2\}=10$. In Case 1, $x^8$ in $\dot{z}_1$ cannot be written as quadratic in $x,z_1, z_2$. In Case 2 and Case 3, $x^7$ in $\dot{z}_1$ cannot be written as quadratic in $x, z_1,$ and $z_2$. In Case 4, $x^3$ and $x^4$ in $\dot{x}$ cannot be written as quadratic in $x, z_1,$ and $z_2$. Thus, $2$ monomial new variables are not enough to quadratize all degree $6$ scalar polynomial ODEs. 
    \end{proof}
    \begin{lemma}\label{canquad}
        All degree $6$ scalar polynomial ODEs can be quadratized by three monomial new variables, $z_1:=x^5, z_2:=x^4, z_3:=x^3$.
    \end{lemma}
    \begin{proof}
        Let 
        \begin{equation*}
            \dot{x}=p_6x^6+p_5x^5+p_4x^4+p_3x^3+p_2x^2+p_1x+p_0
        \end{equation*}
        with $p_6\neq 0$. It follows that for constants $c_i$'s:
        \begin{equation}
            \begin{cases}
            \dot{x}=p_6z_1x+p_5z_1+p_4z_2+p_3z_3+p_2x^2+p_1x+p_0\\
            \dot{z}_1=5x^4\dot{x}= 5p_6z_1^2+5p_5z_1z_2+5p_4c_1z_2^2+5p_4(1-c_1)z_1z_3
            +5p_3z_2z_3+5p_2c_2z_1x\\ \ \ \ \ \ \ \ \ \ \ \ \ \ \ \ \  +5p_2(1-c_2)z_3^2+5p_1c_3z_1 + 5p_1(1-c_3)z_2x + 5p_0c_4z_2 + 5p_0(1-c_4)z_3x\\
            \dot{z}_2 = 4x^3\dot{x} = 4p_6z_1z_2 + 4p_5c_5z_2^2 + 4p_5(1-c_5)z_1z_3 + 4p_4z_2z_3 + 4p_3c_6z_1x + 4p_3 (1-c_6)z_2^2 \\ \ \ \ \ \ \ \ \ \ \ \ \ \ \ \ \ + 4p_2c_7z_1 + 4p_2(1-c_7)z_2x + 4p_1c_8z_2 + 4p_1 (1-c_8)z_3x + 4p_0 z_3
            \\
            \dot{z}_3 = 3x^2\dot{x} = 3p_6c_9z_2^2 + 3p_6(1-c_9)z_1z_3 + 3p_5z_2z_3 + 3p_4c_{10}z_1x + 3p_4 (1-c_{10})z_3^2 \\ \ \ \ \ \ \ \ \ \ \ \ \ \ \ \ \ + 3p_3c_{11}z_1 + 3p_3 (1-c_{11})z_2x + 3p_2c_{12}z_2 + 3p_2(1-c_{12})z_3x + 3p_1z_3 + 3p_0x^2
            \end{cases}
        \end{equation}
    \end{proof}

    \begin{proof}[Proof of Proposition \ref{3,4,5}]
    
        \emph{Part $(i)$.} All degree 3 scalar polynomial ODEs can be quadratized by exactly one new variable, $z:=x^2$.
        
        Let $\dot{x}=p_3x^3+p_2x^2+p_1x+p_0$ where $p_i\in\mathbb{C}$ for $i=0,1,2,3$ and $z:=x^2$.
        It follows that:
        \begin{equation}
            \begin{cases}
            \dot{x}=p_3zx+p_2x^2+p_1x+p_0\\
            \dot{z}=2x\dot{x}=2p_3z^2+2p_2zx+2p_1x^2+2p_0x
            \end{cases}
        \end{equation}
        
        \emph{Part $(ii)$.} All degree 4 scalar polynomial ODEs can be quadratized by exactly one new variable, $z:=x^3$.
        
        Let $\dot{x}=q_4x^4 + q_3x^3 + q_2x^2 + q_1 x + q_0$ where each $q_i\in\mathbb{C}$. By Lemma \ref{lemma:shift}, any degree $4$ scalar polynomial ODE can be uniquely shifted such that the coefficient behind $x^3$ becomes zero. Applying the change of variables $x = y - \tfrac{q_3}{4 \cdot q_4}$, we get that $\dot{y}=q_4y^4 + p_2y^2 + p_1 y + p_0$ for some $p_j$'s $\in\mathbb{C}$. Let $z:=y^3$.
        It follows that:
        \begin{equation}
            \begin{cases}
            \dot{y}=q_4zy + p_2y^2 + p_1 y+p_0\\
            \dot{z} = 3y^2\dot{y}= 3q_4z^2 + 3p_2zy+3p_1z+3p_0y^2
            \end{cases}
        \end{equation}
        
        Part $(iii)$. All  degree 5 scalar  polynomial  ODEs  can  be  quadratized  by  exactly  two  new  variables, $z_1:=x^4$ and $z_2:=x^3$.
        
        Let $\dot{x}=p_5x^5+p_4x^4+p_3x^3+p_2x^2+p_1x+p_0$ where $p_i\in\mathbb{C}$ for $i=0,1,2,3,4, 5$ and $p_5\neq0$. Also, let $z_1:=x^4$ and $z_2:=x^3$. 
        It follows that:
        \begin{equation}
           \begin{cases}
            \dot{x}=p_5z_1x+p_4z_1+p_3z_2+p_2x^2+p_1x+p_0\\
            \dot{z}_1=(z_1')(\dot{x})=4x^3(\dot{x})=4p_5z_1^2+4p_4z_1z_2+4p_3z_2^2+4p_2z_1x+4p_1z_1+4p_0z_2\\
            \dot{z}_2=(z_2')(\dot{x})=3x^2(\dot{x})=3p_5z_1z_2+3p_4z_2^2+3p_3z_1x+3p_2z_1+3p_1z_2+3p_0x^2
            \end{cases}
        \end{equation}
    \end{proof}
    

\section{Computational Techniques}\label{section: Computational Techniques}
    In this section, we describe the computational techniques used in order to gain the intuition for Theorem \textcolor{red}{\ref{theorem}} and find the form of the quadratization presented in Theorem \textcolor{red}{\ref{degree 6}}. 
    The main tool used for our computation was Gr\"obner bases. 
    
    A Gr\"obner basis is a set of multivariate polynomials that has desirable algorithmic properties. It holds that every set of polynomials can be transformed into a Gr\"obner basis. 
    Gr\"obner basis computation is an effective way of reducing or solving systems of equations and generalizes Gaussian elimination and the Euclidean algorithm for polynomials. For more, see \cite{Basis, Grobner}. 
    
    \subsection{Quadratization with One New Variable}
    In this subsection, we will outline the computational experiments used to gain the intuition for Theorem \ref{theorem}. For simplicity, we will focus on degree $5$ scalar polynomial ODE since it is the smallest degree for which Theorem \ref{theorem} can be applied. 
    
    Using Lemma \ref{lemma:shift}, Lemma \ref{variable}, and Lemma \ref{claim}, we introduce the following set-up:
    \begin{equation*}
        \begin{cases}
        \dot{x}= p_5x^5 + p_3x^3 + p_2x^2 + p_1x + p_0
        \\
        z:= x^4 + q_3x^3 + q_2x^2
        \end{cases}
    \end{equation*}
    
    We define our polynomial ring $R=\mathbb{C}[p_0, p_1, p_2, p_3, p_4, q_3, q_2]$. 
    
    We ask the following elimination question: for what values of $p_i$ does there exist values of $q_j$ such that $\dot{x}$ can be written as some linear combination  of $S_1=\{1, x, x^2, z, zx\}$ and $\dot{z}$ can be written as some linear combination of $S_2=\{1, x, x^2, z, zx, z^2\}$? Notice that $z^2$ is not in $S_1$ because $\deg z^2 = 8$, but $\deg \dot{x}=5$. 
    
    In order to answer our question, we produce the following two matrices where each entry is defined by the coefficient behind the term that corresponds to the row in the function that corresponds to the column: 
    \[
        \text{$\dot{x}$ matrix}
        =
        \begin{array}{c}
        
        \begin{array}{c c c c} & \ \ \ 1 \ \ \ x & x^2 \ \  z  \ \ \ xz & \dot{x} \\
        \end{array}
        \\
        \begin{array}{c}
        1 \\
        x \\
        x^2 \\
        x^3 \\
        x^4 \\
        x^5
        \end{array}
        
        \begin{pmatrix}
        1 & 0 & 0 & 0 & 0 & p_0\\
        0 & 1 & 0 & 0 & 0 & p_1\\
        0 & 0 & 1 & q_2 & 0 & p_2\\
        0 & 0 & 0 & q_3 & q_2 & p_3\\
        0 & 0 & 0 & 1 & q_3 & p_4\\
        0 & 0 & 0 & 0 & 1 & 1
        \end{pmatrix}
        
        \end{array}
    \]
    \[
        \text{$\dot{z}$ matrix}
        =
        \begin{array}{c}
        
        \begin{array}{c c c c c c c c c c c c c c c} 1 & x & x^2\ \  z & xz & & \ z^2 & & & & & & \ \ \dot{z} \ & & \\
        \end{array}

        \\
        \begin{array}{c}
        1 \\
        x \\
        x^2 \\
        x^3 \\
        x^4 \\
        x^5 \\
        x^6\\
        x^7 \\
        x^8
        \end{array}
        
        \begin{pmatrix}
        1 & 0 & 0 & 0 & 0 & 0 & 0\\
        0 & 1 & 0 & 0 & 0 & 0 & 0\\
        0 & 0 & 1 & q_2 & 0 & 0 & 0\\
        0 & 0 & 0 & q_3 & q_2 & 0 & (2p_2q_2+3p_1q_3+4p_0) \\
        0 & 0 & 0 & 1 & q_3 & q_2^2  & (2p_3q_2+3p_2q_3+4p_1)\\
        0 & 0 & 0 & 0 & 1 & 2q_2q_3  & (2p_4q_2+3p_3q_3+4p_2)\\
        0 & 0 & 0 & 0 & 0 & q_3^2+2q_2 & (4p_3+3p_4q_3 + 2q_2)\\
        0 & 0 & 0 & 0 & 0 & 2q_3 & (3q_3+4p_4)\\
        0 & 0 & 0 & 0 & 0 & 1 & 4
        \end{pmatrix}
        
        \end{array}
        \]
    Notice that in the $\dot{x}$ matrix, the first $5$ column vectors are linearly independent. 
    Thus, $\dot{x}$ can be written as some linear combination  of $S_1=\{1, x, x^2, z, zx\}$ iff all $6$ column vectors in the $\dot{x}$ matrix are linearly dependent. 
    This happens precisely when the determinant of the $\dot{x}$ matrix equals $0$. 
    
    Similarly, notice that in the $\dot{z}$ matrix, the first $6$ column vectors are linearly independent. 
    Thus, $\dot{z}$ can be written as some linear combination  of $S_2=\{1, x, x^2, z, zx, z^2\}$ iff all $7$ column vectors in the $\dot{z}$ matrix are linearly dependent. 
    Since the $\dot{z}$ matrix in non-square, this happens precisely when the minors of the $\dot{z}$ matrix equal $0$. 
    
    Thus, we define our set of polynomials as the determinant of the $\dot{x}$ matrix and the minors of the $\dot{z}$ matrix in $p_i$'s and $q_j$'s. More precisely, the problem we look to solve is: for what values of $p_i$'s does there exist $q_j$'s such that the determinant of the $\dot{x}$ matrix and the minors of the $\dot{z}$ matrix equal zero. 
    The way we solve this problem is by computing the Gr\"obner basis of this set of polynomials and then, selecting only the expressions in $p_i$'s. While this does not give the complete set of $p_i$'s, it gives us the closure of this set~\cite[Theorems~2 and~3, \S 3.1]{CLO}.
    This was enough to give us the intuition for the proof of Theorem~\ref{theorem}.

    The polynomials only in $p_i$'s of the computed Gr\"obner basis are: 
    \begin{equation*}
        \begin{cases}
        p_5p_0 = 0\\
        p_3 = 0
        \end{cases}
    \end{equation*}
    Since $p_5\neq 0$, it follows that $p_0=p_3=0$. Testing with higher degree scalar polynomial ODEs, the same pattern holds, giving us the intuition for the proof of Theorem \ref{theorem}.

    \subsection{Quadratization with Two New Variables}
    In this subsection, we will focus on the methods we used to find the quadratization presented in Theorem \ref{degree 6}. Our goal was to find the full characterization of the quadratization of a scalar polynomial ODE of degree $6$ with exactly two new variables or determine that it is not possible. 
    
    By Lemma \ref{2 var limitation}, we know one of the new variables has degree $5$. Using Corollary \ref{cor: 2 var shift}, we can say that our second new variable does not have degree $5$. To start, we assume that the other new variable has degree less than $5$. So, we try every possible combination of degrees of variables. In other words, we try two new variables of degree $5$ and degree $4$, degree $5$ and degree $3$, and degree $5$ and degree $2$. By Lemma \ref{claim}, we do not consider degree $1$ and degree $0$ for our second new variable. Here, we will simply outline the correct solution where our first new variable has degree $5$ and our second new variable has degree $3$. However, in order to thoroughly experiment with degree $6$ scalar polynomial ODEs, we conducted this computation with other degrees for our second new variable. We introduce the following setup. 
    
    Using Lemma \ref{lemma:shift}, let
    \begin{equation*}
        \dot{x}=x^6+p_4x^4+p_3x^3+p_2x^2+p_1x+p_0
    \end{equation*}
    Using Lemma  \ref{claim}, we can take our new variables with no linear or constant term. Let
    \begin{equation*}
        z_1:=x^5 + q_4x^4 + q_2x^2, \qquad z_2:=x^3 + r_2x^2
    \end{equation*}
    Note that we used Lemma \ref{simplify} to take $z_1$ with no $x^3$ term. 
    
    We would like to know when $\dot{x}$, $\dot{z}_1$, and $\dot{z}_2$ can be written as at most quadratic combination of $x, z_1, z_2$. In other words, we would like to know when $\dot{x}$, $\dot{z}_1$, and $\dot{z}_2$ can be written as a linear combination of the terms in $S=\{1, x, x^2, z_1, z_1x, z_1^2, z_2, z_2x, z_2^2, z_1z_2\}$. Thus, this happens when:
    \begin{equation}\label{grsystem}
        \begin{cases}
        \dot{x}-c_1x-c_2x^2-c_3z_1-c_4z_1x-c_5z_2-c_6z_2x-c_7z_2^2=0\\
        \dot{z}_1-c_8x-c_9x^2-c_{10}z_1-c_{11}z_1x-c_{12}z_1^2-c_{13}z_2-c_{14}z_2x-c_{15}z_2^2=0\\
        \dot{z}_2-c_{16}x-c_{17}x^2-c_{18}z_1-c_{19}z_1x-c_{20}z_1^2-c_{21}z_2-c_{22}z_2x-c_{23}z_2^2=0
        \end{cases}
    \end{equation}

    Thus, we define our polynomial ring as $R=\mathbb{C}[p_4, p_3, p_2, p_1, p_0, q_2, r_2, \Bar{c}]$ where $\Bar{c}=[c_1, c_2, \dots, c_{23}]$. Our system of equations is the coefficients behind each monomial term $x^i$ in each equation in \eqref{grsystem} set to zero. Precisely, we ask the following question: for what values of $p_i's$ does there exists $c_j's, q_k
    s$, and $r_l's$ that satisfy \eqref{grsystem}.
    
    In order to reduce the complexity of our computation, we computed the coefficients behind each $x^i$ term of each equation of \eqref{grsystem}. Taking these coefficients equal to zero, we aimed to replace as many $c_j's$ as possible in \eqref{grsystem} with terms in $p_4, p_3, p_2, p_1, p_0$ in order to reduce the number of terms in our polynomial ring. For example, the coefficient behind the linear term in the second equation is $c_8 - \tfrac{5p_0p_3}{4}$. Since the left-hand side of each equation must equal zero, we have that  $c_8 - \tfrac{5p_0p_3}{4}=0$. This gives us that $c_8 = \tfrac{5p_0p_3}{4}$. Thus, we replace $c_8$ in our system of equations with $\tfrac{5p_0p_3}{4}$ and remove $c_8$ from our polynomial ring. By doing this process multiple times, we simplify our computation by reducing the number of variables we must work with. After simplifying, we obtain that our Gröbner Basis is a set of zeros. Observing the replacements we made, we find the form of our quadratization. 


\section{Conclusion}
    We have shown in Theorem \textcolor{red}{\ref{theorem}} that a scalar polynomial ODE $\dot{x} = p(x)=a_nx^n+a_{n-1}x^{n-1}+o(x^{n-1})$ with $n \geqslant 5$ can be quadratized using exactly one new variable if and only if $p(x-\frac{a_{n-1}}{n\cdot a_n})=a_nx^n+ax^2+bx$ for some $a,b \in \mathbb{C}$. We have also shown in Theorem \textcolor{red}{\ref{degree 6}} that all degree $6$ scalar polynomial ODEs can be quadratized with two non-monomial new variables. Finally, we have shown that all degree $3$ and $4$ scalar polynomial ODEs can be quadratized with one monomial new variable and all degree $5$ scalar polynomial ODEs can be quadratized with two monomial new variables. 
    These results indicate that adding non-monomial variables may lead to substantially more optimal quadratization than the monomial ones used in the current software.
    They also give basic intuition about how to exploit non-monomiality (e.g., via shifts as in Theorem~\ref{theorem}).
    
    We employed computational techniques that made use of Gröbner Bases to help us gain intuition for Theorem \textcolor{red}{\ref{theorem}} and find the form of the new variables and quadratization in Theorem \textcolor{red}{\ref{degree 6}}. Our code is attached as a separate file. 
    

\section*{Acknowledgements}

The advisor, Gleb Pogudin, has been partially supported by NSF grants DMS-1853482, DMS-1760448, DMS-1853650, CCF-1564132, and CCF-1563942.



\end{document}